\documentclass[12pt,reqno]{amsart}
\usepackage{amssymb}
\usepackage{}
\usepackage{amsfonts}
\usepackage{mathrsfs}
\usepackage{amsmath,amsthm,amssymb,amsfonts,amscd}
\usepackage{mathrsfs}
\usepackage{bbding}
\usepackage{graphicx,latexsym}

\theoremstyle{plain}
\newtheorem{theorem}{Theorem}[section]
\newtheorem{corollary}{Corollary}[section]

\newtheorem{lemma}{Lemma}[section]

\newtheorem{conjecture}{Conjecture}[section]

\theoremstyle{definition}

\def \N {{\mathbb N}}
\def \Z {{\mathbb Z}}

\def\l{\left}
\def\r{\right}
\def\bg{\bigg}
\def\({\bg(}
\def\){\bg)}
\def\t{\text}
\def\f{\frac}

\def\ls{\leq}
\def\gs{\geq}

\def\eq{\equiv}

\begin{document}
\hbox{Preprint}
\medskip

\title{On the set $\{\pi(kn):\ k=1,2,3,\ldots\}$}

\author{Zhi-Wei Sun}
\address{Department of Mathematics \\ Nanjing University \\ Nanjing 210093 \\Reople's Republic of China}
\email{zwsun@nju.edu.cn}
\author{Lilu Zhao}
\address{School of Mathematics \\ Shandong University \\ Jinan  250100 \\Reople's Republic of China}
\email{zhaolilu@sdu.edu.cn}

\begin{abstract} An open conjecture of Z.-W. Sun states that for any integer $n>1$
there is a positive integer $k\le n$ such that $\pi(kn)$ is prime, where $\pi(x)$
denotes the number of primes not exceeding $x$. In this paper, we show that for any positive integer $n$ the set $\{\pi(kn):\ k=1,2,3,\ldots\}$ contains infinitely many $P_2$-numbers which are products of at most two primes. We also prove that under the Bateman--Horn conjecture the set 
$\{\pi(4k):\ k=1,2,3,\ldots\}$ contains infinitely many primes.
\end{abstract}
\thanks{
\newline \indent The first and the second authors are supported by the Natural Science Foundation of China (Grant No. 11971222 and 11871187 respectively).}

\maketitle

\section{Introduction}
\setcounter{lemma}{0}
\setcounter{theorem}{0}
\setcounter{corollary}{0}
\setcounter{remark}{0}
\setcounter{equation}{0}

\medskip

For $x\ge0$, let $\pi(x)$ denote the number of primes not exceeding $x$.
For the asymptotic behavior of the prime-counting function $\pi(x)$, by the Prime Number Theorem we have
$$\pi(x)\sim\f x{\log x}\ \ \ \t{as}\ x\to+\infty.$$
Since there are no simple closed formula for the exact values of $\pi(x)$ with $x>0$,
it is difficult to obtain combinatorial properties of the prime-counting function $\pi(x)$.

In 1962, S. Golomb \cite{G}] proved that for any integer $m>1$ there is an integer $n>1$ with $n/\pi(n)=m$ (i.e., $\pi(n)=n/m$).
In 2017 Z.-W. Sun \cite{S17} obtained the following general result: For any $a\in\Z$ and $m\in\Z^+=\{1,2,3,\ldots\}$, we have
$$\pi(n)=\f{n+a}m\ \ \ \t{for some integer}\ n>1$$
if and only if $a\le s_m$,
where
$$s_m:=\max\{km-p_k:\ k\in\Z^+\}=\max\{km-p_k:\ k=1,2,\ldots,\lfloor e^{m+1}\rfloor\}$$
with $p_k$ the $k$-th prime. This implies that for any integer $m>4$
we have $\pi(mn)=m+n$ for some $n\in\Z^+$ (cf. \cite[Corollary 1.2]{S17}).

On Feb. 9, 2014, Z.-W. Sun \cite{S14} made the following conjecture.

\begin{conjecture} \label{conj-Sun}{\rm (Sun \cite[Conjecture 2.1(i)]{S15})} For any integer $n>1$,
there is a positive integer $k\le n$ such that $\pi(kn)$ is prime.
\end{conjecture}

This conjecture was verified by Sun for all $n=2,3,\ldots,2\times10^7$ (cf. \cite{S14}). For $n=10$, among the ten numbers
\begin{gather*}\pi(10)=4,\ \pi(20)=8,\ \pi(30)=10,\ \pi(40)=12,\ \pi(50)=15,
\\\pi(60)=17,\ \pi(70)=19,\ \pi(80)=22,\ \pi(90)=24,\ \pi(100)=25,
\end{gather*}
only $\pi(60)=17$ and $\pi(70)=19$ are prime. Note also that among the $13$ numbers
\begin{gather*}\pi(13)=6,\ \pi(2\times13)=9,\ \pi(3\times13)=12,\ \pi(4\times13)=15,\ \pi(5\times13)=18,
\\\pi(6\times13)=21,\ \pi(7\times13)=24,\ \pi(8\times13)=27,\ \pi(9\times13)=30,
\\ \pi(10\times13)=31,\ \pi(11\times13)=34,\ \pi(12\times13)=36,\ \pi(13\times 13)=39
\end{gather*}
only $\pi(10\times13)=31$ is prime.

 Motivated by Conjecture \ref{conj-Sun}, for any $n\in \Z^{+}$ we introduce the set
\begin{align}\mathcal{A}_n=\{\pi(kn):\ k\in \Z^{+}\}.\end{align}
Clearly, $\mathcal A_1$ coincides with $\N=\{0,1,2,\ldots\}$, and
$\mathcal A_2=\Z^+$ since $\pi(p_j+1)=j$ for all $j=2,3,\ldots$.
As $\lim_{k\to+\infty}\pi(3k)=+\infty$, and $\pi(3(k+1))-\pi(3k)\in\{0,1\}$ for all $k\in\Z^+$,
we see that $$\mathcal A_3=\{m\in\Z^+:\ m\gs \pi(3)\}=\{2,3,\ldots\}.$$
It is not known whether $\mathcal A_4$ contains infinitely many primes.

Throughout this paper, for any $A\subseteq \Z^{+}$ and $x\ge0$, we define
\begin{equation}\label{A}A(x):=\{a\le x:\ a\in A\}.\end{equation}

Now we present our first theorem.
\begin{theorem}\label{theorem0} Let $S\subseteq \Z^{+}$ with
$$\lim_{x\to+\infty}\frac{|S(x)|}{x/\log x}=+\infty.$$
Then, for any $n\in\Z^+$ the set $\mathcal{A}_n$ contains infinitely many elements of $S$.
 \end{theorem}
%\begin{theorem}\label{theorem1}For any $n\in \Z^{+}$, there are infinitely many $P_2$ numbers in $\mathcal{A}_n$.\end{theorem}

If $S=\{a\in\Z^+:\ a\eq r\pmod m\}$ with $m,r\in\Z^+$, then
$$\lim_{x\to+\infty}\f{|S(x)|}x=\f1m\ \ \t{and hence}\ \ \lim_{x\to+\infty}\f{|S(x)|}{x/\log x}=+\infty.$$
Thus Theorem \ref{theorem0} yields the following corollary.

\begin{corollary} For any $n,m,r\in \Z^{+}$, there are infinitely many $a\in \mathcal{A}_n$ with $a\equiv r\pmod{m}$.\end{corollary}

In contrast, Sun \cite[Conjecture 2.2]{S15} conjectured that for each $n\in\Z^+$ we have $n\mid\pi(kn)$
for some $k=1,\ldots,p_n$, and also $\{\pi(kn):\ k=1,\ldots,2p_n\}$ contains a complete system of residues modulo $n$.

As usual, for any $r\in \Z^{+}$, we call $n\in\Z^+$ a $P_r$-number if it is a product of at most $r$ primes. It is known
(cf. \cite[Theorem 6.4]{T}) that the number of $P_2$-numbers up to $X$ is $\gg \frac{X}{\log X}\log\log X$ for $X>1$. So Theorem \ref{theorem0} has the following consequence.

\begin{corollary} For any $n\in \Z^{+}$, the set $\mathcal{A}_n$ contains infinitely many $P_2$-numbers.\end{corollary}

 The following conjecture extends a conjecture of Hardy and Littlewood concerning twin primes.

\begin{conjecture}[P. T. Bateman and R. A. Horn \cite{BH}]\label{conj} For $N\in\Z^+$ let $V(N)$ denote the number of positive integers $n\le N$ with $4n+1$ and $4n+3$ twin prime. Then
\begin{align*}V(N)=4\mathfrak{S}\frac{N}{\log^2 N}\Big(1+o(1)\Big)\ \ \t{as}\ N\to+\infty,\end{align*}
where the twin prime constant $\mathfrak{S}$ is given by
\begin{align*}\mathfrak{S}=\prod_{p>2}\Big(1-\frac{1}{(p-1)^2}\Big)\approx 0.6601618\end{align*}
with $p$ in the product runs over all odd primes.
\end{conjecture}
%%\begin{align*}\mathfrak{S}=\prod_{p\ge 3}\Big(1-\frac{2}{p}\Big)\Big(1-\frac{1}{p}\Big)^{-2}\end{align*}

Now we state our second theorem.
\begin{theorem}\label{theorem2}Assuming the truth of Conjecture \ref{conj}, there are infinitely many primes in $\mathcal{A}_4$.\end{theorem}

We are going to prove Theorems \ref{theorem0} and \ref{theorem2} in Sections 2 and 3 respectively.

\section{Proof of Theorem \ref{theorem0}}

\setcounter{lemma}{0}
\setcounter{theorem}{0}
\setcounter{corollary}{0}
\setcounter{remark}{0}
\setcounter{equation}{0}

For $A\subseteq \Z^{+}$, we write $A^{c}$ for $\Z^{+}\setminus A$, the complement of $A$.

\begin{lemma}\label{lemma1} For integers $K\ge3$, we have
\begin{equation}|\label{2.1}\mathcal{A}_n^{c}(X)|\ll_n \frac{X}{\log X},\end{equation}
where $X=\pi(Kn)$ with $n\in\Z^+$. \end{lemma}
\begin{proof} Note that
\begin{equation}\label{equation1}\begin{aligned}|\mathcal{A}_n^{c}(X)|=&\l|\l\{a\in\Z^+:\ a\in\bigcup_{k=1}^K(\pi((k-1)n),\pi(kn))\r\}\r|
\\=&\sum_{\substack{1\le k\le K\\ \pi(kn)-\pi(kn-n)\ge 2}}\Big(\pi(kn)-\pi(kn-n)-1\Big).
\end{aligned}\end{equation}
Since $\pi(kn)-\pi(kn-n)-1\le n$, we have
\begin{align*}|\mathcal{A}_n^{c}(X)|\leq n|\mathcal{K}|,\end{align*}
where
\begin{align*}\mathcal{K}=\{1\le k\le K:\ \pi(kn)-\pi(kn-n)\ge2\}.\end{align*}

For each $k\in \mathcal K$, there exist two primes $p$ and $q$  such that
$kn-k<p<q\le kn$ and hence $2\le q-p<n$. Thus
\begin{align*}|\mathcal{K}|\le \sum_{2\le h<n}|\{p\ls Kn:\ p\ \t{and}\ p+h\ \t{are both prime}\}|.\end{align*}
It is well known (cf. \cite[Theorem 6.7]{ik2004}) that
\begin{align*}\pi_h(Kn):=|\{p\ls Kn:\ p\ \t{and}\ p+h\ \t{are both prime}\}|
\ll_{h} \frac{Kn}{\log^2(Kn)},\end{align*}
where the implied constant may depend on $h\ge2$. Combining the above, we obtain
\begin{align}\label{bound}|\mathcal{A}_n^{c}(X)|\leq n|\mathcal{K}|\ll_{n} \frac{K}{\log^2 K}.\end{align}

In view of the Prime Number Theorem,
\begin{align}\label{asym}X= \frac{Kn}{\log(Kn)}\Big(1+o(1)\Big).\end{align}
Now \eqref{2.1} follows from  \eqref{bound} and \eqref{asym}. This concludes the proof.
\end{proof}

\noindent\textit{Proof of Theorem \ref{theorem0}.}
By Lemma \ref{lemma1}, there is a constant $C_n>0$ such that for any $K\in\{3,4,\ldots\}$ we have
\begin{align}\label{bound4}|\mathcal{A}_n^{c}(X)|\le C_n\frac{X}{\log X},\end{align}
where $X=Kn$.
As $\lim_{x\to+\infty}|S(x)|/(x/\log x)=+\infty$, if $K\in\Z^+$ is large enough then
\begin{equation}\label{2Cn}\f{|S(X)|}{X/\log X}\ge2C_n.\end{equation}.

Let $K\in\{3,4,\ldots\}$ be large enough so that \eqref{2Cn} holds.
Then, for $X=\pi(Kn)$ we have
$$\frac{|S_1(X)|}{X/\log X}+\frac{|S_2(X)|}{X/\log X}=\frac{|S(X)|}{X/\log X}\ge 2C_n,$$
where $S_1=S\cap \mathcal{A}_n$ and $S_2=S\cap \mathcal{A}_n^{c}$.
As $$\frac{|S_2(X)|}{X/\log X}\le \frac{|\mathcal{A}_n^{c}(X)|}{X/\log X}\le C_n$$
 by \eqref{bound4}, we obtain
 $$|S_1(X)|\ge C_n\frac{X}{\log X}.$$

 In view of the above, $\lim_{x\to+\infty}|S_1(x)|=+\infty$. So $\mathcal A_n$ contains infinitely many elements of $S$. This completes the proof of Theorem \ref{theorem0}. \qed
\bigskip

\section{Proof of Theorem \ref{theorem2}}

\setcounter{lemma}{0}
\setcounter{theorem}{0}
\setcounter{corollary}{0}
\setcounter{remark}{0}
\setcounter{equation}{0}

\medskip
\noindent\textit{Proof of Theorem \ref{theorem2}.} Let $X=\pi(4K)$ with $K\in\{3,4,\ldots\}$.
Applying \eqref{equation1} with $n=4$, we get
\begin{align*}|\mathcal{A}_4^{c}(X)|=\sum_{\substack{1\le k\le K\\ \pi(4k)-\pi(4k-4)\ge 2}}\Big(\pi(4k)-\pi(4k-4)-1\Big).\end{align*}
For each integer $k>1$, the interval $(4k-4,4k]$ contains at most two primes.
Note also that $\pi(4)-\pi(0)=2$. So we have
\begin{align*}|\mathcal{A}_4^{c}(X)|=1+|\mathcal{V}|,\end{align*}
where\begin{align*}\mathcal{V}=\{1\le k<K:\ \pi(4k+4)-\pi(4k)=2\}.\end{align*}

For any $k=1,\ldots,K-1$, clearly $\pi(4k+4)-\pi(4k)= 2$ if and only if both $4k+1$ and $4k+3$ are twin prime. Under Conjecture \ref{conj}, we have
\begin{align*}|\mathcal{V}|=V(K-1)=4\mathfrak{S}\frac{K}{\log^2 K}\Big(1+o(1)\Big)\end{align*}
and hence
\begin{align*}|\mathcal{A}_4^{c}(X)|=4\mathfrak{S}\frac{K}{\log^2 K}\Big(1+o(1)\Big).\end{align*}
 By the Prime Number Theorem,
\begin{align*}X= \frac{4K}{\log K}\Big(1+o(1)\Big)\ \ \t{and}\ \ \pi(X)=\f{X}{\log X}(1+o(1)).\end{align*}
Thus
\begin{align*}\pi(X)-|\mathcal{A}_4^{c}(X)|=(1-\mathfrak{S})\frac{X}{\log X}\Big(1+o(1)\Big).\end{align*}

Note that $\mathfrak{S}<1$. By the above,
$$|\{p\le X:\ p \ \t{is a prime in}\ \mathcal A_4\}|
\ge\pi(X)-|\mathcal{A}_4^{c}(X)|\to+\infty$$
as $X=4K\to+\infty$.
So $\mathcal{A}_4$ contains infinitely many primes. This concludes the proof. \qed

\end{document}